\documentclass[10pt]{article}
\usepackage{amsfonts,amssymb,amsmath,amsthm}
\usepackage{url}
\usepackage{enumerate}

\newtheorem{theorem}{Theorem}[section]

\newtheorem{lemma}[theorem]{Lemma}

\newtheorem{proposition}[theorem]{Proposition}
\newtheorem{corollary}[theorem]{Corollary}
\theoremstyle{definition}
\newtheorem{definition}[theorem]{Definition}

\numberwithin{equation}{section}

\author{Shohei Satake
\footnote{Graduate School of System Informatics, Kobe University, Rokkodai 1-1, Nada, Kobe, Japan. E-mail: 155x601x@stu.kobe-u.ac.jp.
}
}

\begin{document}

\title{A note on the relation between two properties of random graphs}

\maketitle

\begin{abstract}
The $t$-e.c. and pseudo-random property are typical properties of random graphs. 
In this note, we study the gap between them which has not been studied well. 
As a main result, we give the first explicit construction of infinite families of $t$-e.c. graphs which are not families of best possible pseudo-random graphs.
\end{abstract}

\section{Introduction}
{\it Erd\H{o}s-R\'{e}nyi random graphs} (or {\it random graphs}) are graphs on the vertex set $\{1, 2, \ldots, n\}$ which can be obtained by choosing edges independently with probability $p$ 
(see e.g. \cite{AS16}, \cite{B01}).
The probability $p$ is called {\it edge probability}. 
For a property $\mathcal{P}$, we say that random graphs {\it asymptotically almost surely} ({\it a.a.s}) satisfy $\mathcal{P}$ if the probability of the event that graphs satisfy 
$\mathcal{P}$ tends to $1$ when $n$ goes to infinity. 
One of important research in random graph theory is to investigate typical properties of random graphs which give measures how given deterministic graphs are like random graphs. 
At present, there are several types of such properties.
In this note, as noted in Cameron-Stark~\cite{CS02} and Bonato~\cite{B09}, we mainly focus on the $t$-existentially closed ($t$-e.c.) and pseudo-random property. 
Throughout of this note, a family $(G_i)_{i\geq 1}$ of graphs with $n_i$ vertices means an infinite sequence of graphs such that $n_i \rightarrow \infty$ when $i$ goes to infinity.
For positive functions $f$ and $g$ of positive integers, the notion $f=O(g)$ means that there exists a constant $C>0$ such that $f(i)\leq Cg(i)$ for sufficiently large $i$.
The notion $f=\Omega(g)$ means that there exists a constant $C>0$ such that $f(i)\geq Cg(i)$ for sufficiently large $i$.
Finally, the notion $f=o(g)$ means that $\lim_{i \rightarrow \infty}f(i)/g(i)=0$.
\par Let $t$ be a positive integer. 
A graph is called a {\it $t$-existentially closed} ({\it $t$-e.c.}){\it graph} 
if for any two disjoint subsets of vertex set, say $A$ and $B$, satisfying $|A \cup B|=t$,
there exists a vertex $z \notin A \cup B$ such that $z$ is adjacent to all vertices of $A$ but no vertices of $B$. 
Here $A$ or $B$ may be empty set. 
We also call this adjacency property the {\it $t$-e.c. property}. 
A simple probabilistic argument shows that for each constant $0<p<1$, random graphs with edge probability $p$ or $p \pm o(1)$ a.a.s. satisfy the $t$-e.c. property for any $t \geq1$. 
Moreover, for $t$-e.c. graphs, the more the value $t$ is, the more the graph is like random graphs.
\par The {\it pseudo-random property} is explained by the notion of ^^ ^^ jumbled graphs'' or ^^ ^^ bi-jumbled graphs'' which come from the works by Thomason~\cite{T87a}, \cite{T87s} and Kohayakawa et.al.~\cite{KRSSS07}.
In this note, we focus on the notion of bi-jumbled graphs which is actually more general than the one of jumbled graphs.
Let $0<p<1\leq \alpha$. 
Then a graph $G$ is called a {\it $(p, \alpha)$-bi-jumbled graph}
if for any subsets $U$ and $W$ of vertex set of $G$, 
\begin{align}
\label{eq:jumb}
\Bigl|e_{G}(U, W)-p\cdot |U|\cdot |W| \Bigr| \leq \alpha \cdot \sqrt{|U|\cdot |W|},
\end{align}
where $e_{G}(U, W)$ is the number of edges of $G$ such that one endpoint is in $U$ and the another one is in $W$. 
(Here each edge whose endpoints in $U\cap W$ is counted twice in $e_G(U, W)$).
It is well known that random graphs with $n$ vertices and edge probability $p_n$ are a.a.s. $(p_n, \alpha_n)$-bi-jumbled with $\alpha_n=O(\sqrt{n \cdot p_n})$ if $np_n \gg \log p_n$.
(see e.g. \cite[Section 2]{KRSSS07} and  \cite[Section 2.2]{KS06}).
Furthermore, Erd\H{o}s-Spencer~\cite{ES71} showed that the above magnitude of $\alpha_n$ is best possible as long as $p_n(1-p_n) \geq 1/n$ (see also \cite[Section 2]{KRSSS07}).
In this note, we call a family $(G_i)_{i\geq 1}$ {\it a family of best pseudo-random graphs} (with respect to edge probability $p_i$) if $G_i$ is $(p_i, O(\sqrt{n_ip_i}))$-bi-jumbled.
We also call the above property the {\it best pseudo-random property}. 
\par At present, the relation between the $t$-e.c. and pseudo-random property does not seem to be sufficiently clarified but some related observations have been obtained.
For example, Cameron-Stark~\cite{CS02} raised examples of families of graphs which are families of best pseudo-random graphs with respect to edge probability $1/2-o(1)$ but not $t$-e.c. for any $t \geq 4$.
This means that the best pseudo-random property does not necessarily imply the $t$-e.c. property.
We note that families of best pseudo-random graphs which are triangle-free are such examples since $t$-e.c. graphs must contain triangles whenever $t \geq 2$.
At present, various such families have been obtained (see e.g. \cite{A94}, \cite{S19}).
On the other hand, almost no results on the converse direction have been obtained.
Indeed, almost all known explicit families of $t$-e.c. graphs are also families of best pseudo-random graphs and other several families are quite unclear whether they are families of best pseudo-random graphs or not (see e.g. \cite{B09}). 
It is also non-trivial to construct infinite families of $t$-e.c. graphs for large $t$. 
\par In this note, for any $t\geq 1$, we explicitly construct many families of $t$-e.c. graphs but not families of best pseudo-random graphs with respect to edge probability $1/2-o(1)$.
We also show that the constructed families are actually special classes of {\it quadratic unitary Cayley graphs} proposed by de Beaudrap~\cite{B10} and Liu-Zhou~\cite{LZ15}. 
\par The rest of this note is organized as follows.
In Section~\ref{sec:prop}, we explain that the pseudo-random property of regular graphs can be described by graph eigenvalues.
In Section~\ref{sec:const}, we give a construction providing many families of $t$-e.c. graphs which are not families of best pseudo-random graphs for every $t \geq 1$ without probabilistic arguments.
In Section~\ref{sec:conclusion}, we make some concluding remarks.

\section{The pseudo-random property of regular graphs and graph eigenvalues}
\label{sec:prop}
In this section, we explain that the pseudo-random property of families of regular graphs can be described by eigenvalues of their adjacency matrices.
The {\it adjacency matrix} of a graph on the vertex set $\{1, 2, \ldots, n\}$ is the $(0, 1)$-square matrix of order $n$ such that the $(i, j)$-entry is $1$ if and only if $i$ and $j$ are adjacent.
Let $(G_i)_{i \geq 1}$ be a family of regular graphs with $n_i$ vertices such that the degree of $G_i$ is $d_i$. 
Let $d_i=\lambda_1(G_i) \geq \lambda_2(G_i) \geq \cdots \geq \lambda_{n_i}(G_i)$ 
are eigenvalues of its adjacency matrix.
We define $\lambda(G_i):=\max\{\lambda_2(G_i), -\lambda_{n_i}(G_i)\}$.
Then the following two lemmas hold.
 
\begin{lemma}[Expander-mixing lemma, e.g. Alon-Spencer~\cite{AS16}]
\label{lem:exp}
Then for any subsets $U$, $W$ of vertex set of $G_i$, 
\begin{equation}
\Bigl|e_{G_i}(U, W)-\frac{d_i}{n_i}\cdot|U|\cdot |W| \Bigr|\leq \lambda(G_i) \cdot \sqrt{|U|\cdot |W|}.
\end{equation}
\end{lemma}

\begin{lemma}[A converse of the expander-mixing lemma, Bilu-Linial~\cite{BL06}]
\label{lem:cexp}
Suppose that for any disjoint subsets $U$, $W$ of vertex set of $G_i$, there exists $\rho_i>0$ such that
\begin{equation}
\Bigl|e_{G_i}(U, W)-\frac{d_i}{n_i} \cdot |U|\cdot |W| \Bigr|\leq \rho_i \cdot \sqrt{|U|\cdot |W|}.
\end{equation}
Then, $\lambda(G_i)=O(\rho_i(1+\log(d_i/\rho_i))$. 
\end{lemma}
From these lemmas, we obtain the following corollary.
\begin{corollary}
\label{cor:exp}
\begin{enumerate}
\item[(1)] For each $i \geq 1$, $G_i$ is $(d_i/n_i, \lambda(G_i))$-bi-jumbled. Thus, $(G_i)_{i \geq 1}$ is a family of best pseudo-random graphs with respect to edge probability $d_i/n_i$ if $\lambda(G_i)=O(\sqrt{d_i})$. 
\item[(2)] Let $\varepsilon>0$ be a fixed real number.
 Assume that $d_i \rightarrow \infty$ when $i \rightarrow \infty$ and $\lambda(G_i)= \Omega(d_i^{1/2+\varepsilon})$. 
Then, $(G_i)_{i \geq 1}$ is not a family of best pseudo-random graphs with respect to edge probability $d_i/n_i$.
\end{enumerate} 
\end{corollary}
\begin{proof}
The statement of $(1)$ is a direct consequence of Lemma~\ref{lem:exp}.
We prove $(2)$.
Now assume that $(G_i)_{i \geq 1}$ is a family of best pseudo-random graphs. Then by Lemma~\ref{lem:cexp}, 
$\lambda(G_i)=O(\sqrt{d_i} \cdot \log d_i)$, which contradicts the assumption of (2). 
\end{proof}
We note that if $d_i \leq (1-\delta)n_i$ for some fixed $\delta>0$, then $\lambda(G_i)\geq \Omega(\sqrt{d_i})$ (see e.g. \cite[Section 2.4]{KS06}).
Thus a family of best pseudo-random regular graphs are best possible up to constant in the sense of graph eigenvalues. 
　

\section{An explicit construction}
\label{sec:const}
In this section, we prove the following main theorem. 
\begin{theorem}
\label{thm:mainthm}
For each integer $t \geq 1$ and odd integer $e\geq 3$, let $(q_i)_{i \geq 1}$ be the sequence of consecutive Pythagorean primes such that 
$q_1$ is the least Pythagorean prime satisfying 
\[
q_1^e-(t2^{t-1}-2^t+1)q_1^{e-\frac{1}{2}}-t2^tq_1^{e-1}+t2^{t-1}>0.
\]
Then one can explicitly construct a family $(G_i)_{i \geq 1}$ of $t$-e.c. regular graphs which is not a family of best pseudo-random graphs with respect to edge probability $1/2-o(1)$.
Here $G_i$ is a $(q_i^e-q_i^{e-1})/2$-regular graph with $q_i^e$ vertices.
\end{theorem}
\noindent This theorem can be directly obtained from Theorem~\ref{thm:main} and \ref{thm:LZ15} shown later. 
\par First, we explain the construction.
Let $q$ be a Pythagorean prime, that is, a prime of the form $q \equiv 1 \pmod{4}$. 
By the Dirichlet's theorem, there are infinitely many such primes.
Let $e\geq 1$ be an odd integer. 
Then we construct Cayley graphs over the additive group of the residue ring $\mathbb{Z}_{q^e}:=\mathbb{Z}/q^e\mathbb{Z}$ as follows.

\begin{definition}
\label{def:const}
The graph $G_{q^e}$ is the graph with vertex set $\mathbb{Z}_{q^e}$ and 
edge set $\{\{x, y\} \mid \chi_{q^e}(x-y)=1\}$, where $\chi_{q^e}(x):=(\frac{x}{q^e})=(\frac{x}{q})^e$, where $(\frac{x}{q})$ is the Legendre symbol.
\end{definition}
Since $q \equiv 1 \pmod{4}$, $G_{q^e}$ is well-defined.   
When $e=1$, the graph $G_{q}$ is the Paley graphs with $q$ vertices.
By the following proposition, we see that $G_{q^e}$ is a special type of quadratic unitary Cayley graphs proposed by de Beaudrap~\cite{B10} and Liu-Zhou~\cite{LZ15}.

\begin{definition}
Let $R$ be a finite commutative ring  and let $Q_R$ be the set of squares of unit elements of $R$.
Let $T_R:=Q_R \cup (-Q_R)$.
The {\it quadratic unitary Cayley graph} $Cay(R, T_R)$ is the graph with vertex set $R$ such that two elements $x$ and $y$ are adjacent if and only if $x-y \in T_R$.
\end{definition}

\begin{proposition}
\label{prop:square}
$G_{q^e}$ is the quadratic unitary Cayley graph $Cay(\mathbb{Z}_{q^e}, T_{\mathbb{Z}_{q^e}})$.
\end{proposition}
\begin{proof}
By the definition of $G_{q^e}$, two distinct vertices $x$ and $y$ are adjacent in $G_{q^e}$ if and only if $\chi_{q^e}(x-y)=1$.
Since $e$ is odd, $\chi_{q^e}(x-y)=1$ if and only if $(\frac{x-y}{q})=1$, that is, $x-y$ is a nonzero square modulo $q$. 
From the Hensel's lemma (see e.g. \cite[Chapter 13]{W12}), $x-y$ is a nonzero square modulo $q$ if and only if $x-y \in Q_{\mathbb{Z}_{q^e}}$.
Finally, since $q \equiv 1 \pmod{4}$, $-1$ is a nonzero square modulo $q$ and so $-1 \in Q_{\mathbb{Z}_{q^e}}$, which implies that $T_{\mathbb{Z}_{q^e}}=Q_{\mathbb{Z}_{q^e}}$.
\end{proof}

By Proposition~\ref{prop:square}, we see the following proposition.

\begin{proposition}
\label{prop:deg}
\begin{enumerate}
\item[$(1)$] $G_{q^e}$ has $q^e$ vertices.
\item[$(2)$] $G_{q^e}$ is a $(q^e-q^{e-1})/2$-regular graph.
\end{enumerate}
\end{proposition}
\begin{proof}
$(1)$ is directly obtained from Definition~\ref{def:const}.
We prove $(2)$.
By Proposition~\ref{prop:square}, we see that $G_{q^e}$ is $|T_{\mathbb{Z}_{q^e}}|$-regular and so we shall compute the size of $T_{\mathbb{Z}_{q^e}}=Q_{\mathbb{Z}_{q^e}}$.
Let $\mathbb{Z}_{q^e}^*$ be the unit group of $\mathbb{Z}_{q^e}$. 
Note that $\mathbb{Z}_{q^e}^*$ is the cyclic group of order $\varphi(q^e)=q^e-q^{e-1}$ where $\varphi$ is the Euler's totient function.
Let $x$ be a generator of $\mathbb{Z}_{q^e}^*$.
Clearly, $Q_{\mathbb{Z}_{q^e}}=\{x^{2a} \mid 1 \leq a \leq (q^e-q^{e-1})/2\}$, completing the proof.
\end{proof}
Now we prove the following theorem.
\begin{theorem}
\label{thm:main}
For every $t \geq 1$, $G_{q^e}$ is $t$-e.c. if $q$ and $e$ satisfy
\begin{align}
\label{eq:main}
q^e-(t2^{t-1}-2^t+1)q^{e-\frac{1}{2}}-t2^tq^{e-1}+t2^{t-1}>0.
\end{align}
\end{theorem}
To prove the Theorem~\ref{thm:main}, we use a slight generalization of the method in \cite{BEH81}, \cite{BT81} and \cite{GS71} using character sums estimation over finite fields. 
This method was found to prove that Paley graphs are $t$-e.c. graphs.
Similar methods were also applied for other Cayley graphs over finite fields in e.g. \cite{AC06} and \cite{KP04}.   
Based on their discussion, we shall prove that
\begin{equation}
\label{eq:prf1}
f(A, B):=\sum_{z \in \mathbb{Z}_{q^e}\setminus Z_{A, B}} \prod_{a\in A}\{1+\chi_{q^e}(z-a)\}\prod_{b\in B}\{1-\chi_{q^e}(z-b)\} >0
\end{equation}
for all disjoint subsets $A, B \subset \mathbb{Z}_{q^e}$ such that $|A\cup B|=t$ if (\ref{eq:main}) holds.
Here $Z_{A,B}$ is the set of elements $z$ such that $z-c=qv$ for some $c \in A \cup B$ and $v \in \mathbb{Z}_{q^e}$. 
Remark that, in the range of $z$ in the first sum, we must exclude the elements of $Z_{A, B}$ since, if $z-c=qv$ for some $c \in A \cup B$ and $v \in \mathbb{Z}_{q^e}$, then 
$z$ cannot satisfy the definition of the $t$-e.c. property. 
In fact, if so, from the definition of $\chi_{q^e}$, $z$ cannot be adjacent to any $c \in A\cup B$ in $G_{q^e}$.
Now let $Z_{A,B}^*=Z_{A,B} \setminus  (A\cup B)$ and 
\[
g(A, B):=\sum_{z \in \mathbb{Z}_{q^e}\setminus Z_{A, B}^*} \prod_{a\in A}\{1+\chi_{q^e}(z-a)\}\prod_{b\in B}\{1-\chi_{q^e}(z-b)\}.
\]
Note that, in the range of $z$ in the first sum, the set $A \cup B$ is added. 
To obtain (\ref{eq:prf1}), we shall obtain a lower bound of $g(A, B)$. 
To explain why, let 
\[
h(A, B):=\sum_{z \in A \cup B} \prod_{a\in A}\{1+\chi_{q^e}(z-a)\}\prod_{b\in B}\{1-\chi_{q^e}(z-b)\}.
\]
Then we can easily see that 
\begin{equation}
\label{eq:h(A,B)}
h(A,B)\leq t2^{t-1}.
\end{equation} 
We also see that
\begin{equation}
\label{eq:fgh}
f(A,B)=g(A,B)-h(A,B) 
\end{equation}
since $\mathbb{Z}_{q^e}\setminus Z_{A, B}=(\mathbb{Z}_{q^e}\setminus Z_{A, B}^*) \setminus (A\cup B)$. 
So, by combining that lower bound of $g(A, B)$, (\ref{eq:h(A,B)}) and (\ref{eq:fgh}), we will get (\ref{eq:prf1}).
To get a lower bound of $g(A, B)$, at first, we give the following character sum estimation over $\mathbb{Z}_{q^e}$ by combining a known character sum estimation elementary number-theoretic observations.  
\begin{lemma}
\label{lem:weil}
Let $k\geq 1$ be a integer and $a_1, a_2, \ldots, a_k$ be distinct elements of $\mathbb{Z}_{q^e}$.
Then, 
\begin{align}
\label{eq:weil}
\Bigl|\sum_{x \in \mathbb{Z}_{q^e}}\chi_{q^e}(x-a_1)\cdots \chi_{q^e}(x-a_k) \Bigr| \leq (k-1)q^{e-\frac{1}{2}}.
\end{align}
\end{lemma}

\begin{proof}[Proof of Lemma~\ref{lem:weil}]
We shall prove that 
\begin{align}
\label{eq:weil2}
\sum_{x \in \mathbb{Z}_{q^e}}\chi_{q^e}(x-a_1)\cdots \chi_{q^e}(x-a_k)=q^{e-1}\sum_{x \in \mathbb{Z}_q}\chi_{q}(x-a_1)\cdots \chi_{q}(x-a_k)
\end{align}
since we can use the following Burgess's estimation (see e.g. \cite[Chapter II.2]{S76}); 
\begin{align}
\label{eq:burgess}
\Bigl|\sum_{x \in \mathbb{Z}_q}\chi_q(x-a_1)\cdots \chi_{q}(x-a_k) \Bigr| \leq (k-1)\sqrt{q}.
\end{align}
First, $\chi_{q^e}$ is a Dirichlet character modulo $q^e$ of conductor $q$, that is,
$
\chi_{q^e}(x)=\chi_{q^e}(y)
$
whenever $x \equiv y \pmod{q}$. 
So $\chi_{q^e}$ can be regarded as the primitive Dirichlet character $\chi_{q}$ modulo $q$. 
Next observe that, for any $x \in \mathbb{Z}_{q^e}$, there uniquely exist $a_0, a_1, \ldots, a_{e-1} \in \mathbb{Z}_q$ such that $x=a_0+a_1q+a_2q^2+\cdots+a_{e-1}q^{e-1}$.
Therefore, for any $a \in \mathbb{Z}_{q}$, there are $q^{e-1}$ elements $x \in \mathbb{Z}_{q^e}$ such that $\chi_{q^e}(x)=\chi_{q^e}(a)$, completing the proof. 
\end{proof}
Now we can get the following lower bound of $g(A, B)$. 
\begin{lemma}
\label{lem:g(A,B)}
\begin{align}
\label{eq:prf2}
g(A,B) \geq q^e-(t2^{t-1}-2^t+1)q^{e-\frac{1}{2}}-t2^{t}q^{e-1}+t2^{t}.
\end{align}
\end{lemma}
\begin{proof}[Proof of Lemma~\ref{lem:g(A,B)}]
First, we obtain that
\begin{equation}
\label{eq:prf3}
\sum_{z \in \mathbb{Z}_{q^e}\setminus Z_{A, B}^*} 1\geq q^e-tq^{e-1}+t
\end{equation}
since $|Z_{A, B}| \leq t(q^e-\phi(q^e))=tq^{e-1}$ and
$\mathbb{Z}_{q^e}\setminus Z_{A, B}^*$ contains $A \cup B$.
\par Now let $A\cup B=\{c_1, c_2, \ldots, c_t\}$. 
From the definition of $g(A, B)$ and the triangle inequality, we see that
\begin{align}
\label{eq:prf4}
\begin{split}
\biggl|g(A, B)-\sum_{z \in \mathbb{Z}_{q^e}\setminus Z_{A, B}^*} 1 \biggr|
=\sum_{1 \leq k \leq t}\; \sum_{1 \leq i_1<i_2<\cdots<i_k \leq t} \; \biggl|\sum_{z \in \mathbb{Z}_{q^e}\setminus Z_{A, B}^*} \chi_{q^e}(z-c_{i_1})\cdots\chi_{q^e}(z-c_{i_k}) \biggr|.
\end{split}
\end{align}
For each $1 \leq k \leq t$ and $1 \leq i_1<i_2<\cdots<i_k \leq t$, we obtain 
\begin{equation}
\label{eq:prf5}
\biggl|\sum_{z \in \mathbb{Z}_{q^e}\setminus Z_{A, B}^*} \chi_{q^e}(z-c_{i_1})\cdots\chi_{q^e}(z-c_{i_k}) \biggr| \leq (k-1)q^{e-\frac{1}{2}}+tq^{e-1}-t.
\end{equation}
In fact, we get (\ref{eq:prf5}) since 
\begin{align*}
\sum_{z \in \mathbb{Z}_{q^e}\setminus Z_{A, B}^*} \chi_{q^e}(z-c_{i_1})\cdots\chi_{q^e}(z-c_{i_k})
&=\sum_{z \in \mathbb{Z}_{q^e}} \chi_{q^e}(z-c_{i_1})\cdots\chi_{q^e}(z-c_{i_k})\\
&-\sum_{z \in Z_{A, B}^*} \chi_{q^e}(z-c_{i_1})\cdots\chi_{q^e}(z-c_{i_k})
\end{align*}
and from Lemma~\ref{lem:weil} and the fact that $|Z_{A, B}^*|=|Z_{A, B}|-|A \cup B|\leq tq^{e-1}-t$.
Thus, by (\ref{eq:prf4}) and (\ref{eq:prf5}), 
\begin{align}
\label{eq:prf6}
\begin{split}
\biggl|g(A, B)-\sum_{z \in \mathbb{Z}_{q^e}\setminus Z_{A, B}^*} 1 \biggr| 
&\leq \sum_{1\leq k \leq t}\binom{t}{k}\{(k-1)q^{e-\frac{1}{2}}+tq^{e-1}-t\}\\
&=q^{e-\frac{1}{2}}t\sum_{0\leq k \leq t-1}\binom{t-1}{k}+(tq^{e-1}-t-q^{e-\frac{1}{2}})\sum_{1\leq k \leq t}\binom{t}{k}\\
&=t2^{t-1}q^{e-\frac{1}{2}}+(2^{t}-1)(tq^{e-1}-t-q^{e-\frac{1}{2}})\\
&=(t2^{t-1}-2^t+1)q^{e-\frac{1}{2}}+t(2^{t}-1)q^{e-1}-t(2^{t}-1).
\end{split}
\end{align}
By (\ref{eq:prf3}) and (\ref{eq:prf6}), we get (\ref{eq:prf2}).
\end{proof}

Now we are ready to prove the theorem.
\begin{proof}[Proof of Theorem~\ref{thm:main}]
By combining Lemma~\ref{lem:g(A,B)}, (\ref{eq:h(A,B)}) and (\ref{eq:fgh}),
\begin{align*}
f(A,B)&=g(A,B)-h(A,B) \\
&\geq q^e-(t2^{t-1}-2^t+1)q^{e-\frac{1}{2}}-t2^tq^{e-1}+t2^{t}-t2^{t-1}\\
&=q^e-(t2^{t-1}-2^t+1)q^{e-\frac{1}{2}}-t2^tq^{e-1}+t2^{t-1}.
\end{align*}
Thus (\ref{eq:prf1}) holds if (\ref{eq:main}) is satisfied. 
\end{proof}
On the other hand, we can completely determine the spectrum of $G_{q^e}$.
The following theorem is a corollary of Theorem 2.4 in Liu-Zhou~\cite{LZ15}.
Here we give an alternative proof based on elementary calculations and the known fact of quadratic Gauss sums.

\begin{theorem}
\label{thm:LZ15}
The eigenvalues of the adjacency matrix of $G_{q^e}$ are 
\[
\frac{q^e-q^{e-1}}{2}, \; \frac{-q^{e-1}+q^{e-\frac{1}{2}}}{2}, \; 0,\; \frac{-q^{e-1}-q^{e-\frac{1}{2}}}{2}.
\]
Especially, 
\[
\lambda(G_{q^e})=\frac{q^{e-\frac{1}{2}}+q^{e-1}}{2}.
\]
\end{theorem}

\begin{proof}
As we proved in Proposition~\ref{prop:square}, $G_{q^e}$ is the Cayley graph over $\mathbb{Z}_{q^e}$ defined by the subset $T_{\mathbb{Z}_{q^e}}=Q_{\mathbb{Z}_{q^e}}$. 
It is not so difficult to show that the multi-set of eigenvalues of the adjacency matrix is 
\[
\biggl \{\sum_{s \in Q_{\mathbb{Z}_{q^e}}}\exp \biggl(\frac{2\pi ias}{q^e} \biggr) \mid a \in \mathbb{Z}_{q^e} \biggr\}
\]
where $i=\sqrt{-1}$.
Moreover, by the definition of $Q_{\mathbb{Z}_{q^e}}$, 
\begin{equation}
\label{eq:gauss1}
\sum_{s \in Q_{\mathbb{Z}_{q^e}}}\exp \biggl(\frac{2\pi ias}{q^e} \biggr)=\frac{1}{2} \sum_{x \in \mathbb{Z}_{q^e}^*} \exp\biggl(\frac{2\pi ia x^2}{q^e} \biggr).
\end{equation}
Thus we shall evaluate the exponential sum in the right-hand side in (\ref{eq:gauss1}) for $a \neq 0$.
Remark that the sum for $a=0$ is the largest eigenvalue, that is, the degree of $G_{q^e}$. 
Now we see that
\begin{equation}
\label{eq:gauss2}
\sum_{x \in \mathbb{Z}_{q^e}^*} \exp\biggl(\frac{2\pi ia x^2}{q^e} \biggr)
=\sum_{x \in \mathbb{Z}_{q^e}} \exp\biggl(\frac{2\pi ia x^2}{q^e} \biggr)
-\sum_{x \in \mathbb{Z}_{q^e} \setminus \mathbb{Z}_{q^e}^*} \exp\biggl(\frac{2\pi ia x^2}{q^e} \biggr).
\end{equation}
The first exponential sum in the right-hand side in (\ref{eq:gauss2}) is called a {\it quadratic Gauss sum}. 
For $1 \leq k \leq e$, if $a=bq^{e-k}$ for some $1 \leq b \leq q^k$ s.t. $(b, q)=1$, then 
\begin{align}
\begin{split}
\label{eq:gauss3}
\sum_{x \in \mathbb{Z}_{q^e}} \exp\biggl(\frac{2\pi ia x^2}{q^e} \biggr)&=\sum_{x \in \mathbb{Z}_{q^e}} \exp\biggl(\frac{2\pi ibq^{e-k} x^2}{q^e} \biggr)\\
&=q^{e-k}\sum_{x \in \mathbb{Z}_{q^k}} \exp\biggl(\frac{2\pi ibx^2}{q^k} \biggr)
=\biggl(\frac{b}{q^k} \biggr) q^{e-\frac{k}{2}}.
\end{split}
\end{align}
Here we use the following well known fact (see e.g. \cite[Theorem 1.5.2]{BEW98}); for $(b, q)=1$.
\begin{equation}
\sum_{x \in \mathbb{Z}_{q^k}} \exp\biggl(\frac{2\pi ibx^2}{q^{k}}\biggr)=\biggl(\frac{b}{q^k} \biggr) \sqrt{q^k}
\end{equation}
where $(\frac{\cdot}{\cdot})$ is the Jacobi symbol. 
On the other hand, if $x \in \mathbb{Z}_{q^e} \setminus \mathbb{Z}_{q^e}^*$, then $q|x$. Then,
\begin{align}
\label{eq:gauss4}
\begin{split}
\sum_{x \in \mathbb{Z}_{q^e} \setminus \mathbb{Z}_{q^e}^*} \exp\biggl(\frac{2\pi ia x^2}{q^e} \biggr)
&=\sum_{x=qy,\; y\in \mathbb{Z}_{q^{e-1}}} \exp\biggl(\frac{2\pi ia x^2}{q^e} \biggr)\\
&=\sum_{y\in \mathbb{Z}_{q^{e-1}}} \exp\biggl(\frac{2\pi iay^2}{q^{e-2}} \biggr)\\
&=p\sum_{y\in \mathbb{Z}_{q^{e-2}}} \exp\biggl(\frac{2\pi iay^2}{q^{e-2}} \biggr).
\end{split}
\end{align}
Thus, by (\ref{eq:gauss4}) and the discussion to obtain (\ref{eq:gauss3}), we get
\begin{align}
\label{eq:gauss5}
\begin{split}
\sum_{x \in \mathbb{Z}_{q^e} \setminus \mathbb{Z}_{q^e}^*} \exp\biggl(\frac{2\pi ia x^2}{q^e} \biggr)
=
\begin{cases}
q^{e-1}  &\text{$a=bq^{e-2}$, $b \neq 0$;}\\
\Bigl(\frac{b}{q^k} \Bigr) q^{e-1-\frac{k}{2}} & \text{$a=bq^{e-2-k}$, $(b, q)=1$, $1 \leq k \leq e-2$.}\\
\end{cases}
\end{split}
\end{align}
Thus, by (\ref{eq:gauss2}), (\ref{eq:gauss3}), (\ref{eq:gauss4}) and (\ref{eq:gauss5}), we obtain 
\begin{align}
\label{eq:gauss6}
\begin{split}
\sum_{x \in S} \exp\biggl(\frac{2\pi ia x^2}{q^e} \biggr)
=
\begin{cases}
\frac{q^e-q^e-1}{2}  &\text{$a=0$;}\\
\frac{\bigl(\frac{b}{q} \bigr) q^{e-\frac{1}{2}}-q^{e-1}}{2} & \text{$a=bq^{e-1}$, $(b, q)=1$;}\\
\frac{\bigl(\frac{b}{q^2} \bigr) q^{e-1}-q^{e-1}}{2} & \text{$a=bq^{e-2}$, $(b, q)=1$;}\\
\frac{\bigl(\frac{b}{q^{k+2}} \bigr) q^{e-1-\frac{k}{2}}-\bigl(\frac{b}{q^{k}} \bigr) q^{e-1-\frac{k}{2}}}{2} & \text{$a=bq^{e-2-k}$, $(b, q)=1$, $1 \leq k \leq e-2$.}\\
\end{cases}
\end{split}
\end{align}
It is easy to see that $\bigl(\frac{b}{q^2} \bigr)=1$ for every $q$ and $b$, and $\bigl(\frac{b}{q^{k+2}} \bigr)=\bigl(\frac{b}{q^{k}} \bigr)$ for every $q$, $b$ and $k$.
Thus, by (\ref{eq:gauss6}), we prove the theorem.
\end{proof}

\begin{proof}[Proof of Theorem~\ref{thm:mainthm}]
By Proposition~\ref{prop:deg} and Theorem~\ref{thm:LZ15}, we see that for each odd $e \geq 1$, $\lambda(G_{q^e})=\Omega(d_{q^e}^{1/2+\varepsilon})$ where
$d_{q^e}=(q^e-q^{e-1})/2$ is the degree of $G_{q^e}$ and $\varepsilon=(e-1)/2$.
Thus we obtain the theorem by Corollary~\ref{cor:exp} (2) and Theorem~\ref{thm:main}.
\end{proof}

\section{Concluding remarks}
\label{sec:conclusion}
In this note, we constructed families of $t$-e.c. graphs which are not families of best pseudo-random graphs with respect to edge probability $1/2-o(1)$.
At present, we do not know anything for cases of other edge probability.
It would be interesting to investigate such cases. 
\par Below we make some related concluding remarks. 
First, we remark on the relation between $t$-e.c. graphs and expander graphs. 
Here we define expander graphs following the manner in \cite{HLW06}. 
For a graph $G=(V, E)$, the {\it Cheeger constant} $h(G)$ of $G$ is defined as follows. 
\[
h(G):=\min\Bigl\{\frac{e(S, V\setminus S)|}{|S|} \mid Y\subset V,\; |S|\leq \frac{|V|}{2} \Bigr\}.
\]
We call a family $(G_i)_{i \geq 1}$ of graphs {\it a family of expander graphs} if $h(G_i) \geq \varepsilon$ for some $\varepsilon>0$ and any $i \geq 1$. 
For a $d$-regular graph $G$ on $n$ vertices, the Cheeger inequality shows that $h(G)\geq (d-\lambda_2(G))/2$ (see e.g. \cite{HLW06}). 
So a family $(G_i)_{i \geq 1}$ of $d_i$-regular graphs whose the {\it spectral gap} $d_i-\lambda_2(G_i)$ is not zero (or large) for any $i$ is a family of expander graphs. 
Especially a family of best pseudo-random regular graphs with respect to edge probability $d_i/n_i$ is a family of expander graphs with optimal spectral gaps (up to constant).
Theorem~\ref{thm:mainthm} implies that for any $t$, $t$-e.c. graphs do not necessarily ensure that they are the expander graphs with optimum spectral gaps (up to constant). 
\par Next, there is another typical properties of random graphs called the {\it quasi-random property} which was found by Chung-Graham-Wilson~\cite{CGW89}.
They showed the mutually equivalence of some properties which random graphs with constant edge probability $p$ a.a.s. satisfy. 
Here we refer the following property denoted by $P_3$ in \cite{CGW89} as the quasi-random property.
For a constant $0<p<1$, a family of graphs $\{G_i\}_{i\geq 1}$ with $n_i$ vertices has the {\it quasi-random property with respect to edge probability $p$} if 
\[
e(G_i)\geq pn_i^2+o(n_i^2), \; \lambda_1(G_i)=(1 + o(1))pn_i, \; \lambda_2(G_i)=o(n_i).
\]
Here $e(G_i)$ is the size of edge set of $G_i$.
From Proposition~\ref{prop:deg} and Theorem~\ref{thm:LZ15}, the constructed families satisfy the quasi-random property with respect to edge probability around $1/2$.
\par At last, in the research of $t$-e.c. graphs, one of main problems is constructing a family of $t$-e.c. graphs with $O(t2^{t})$ vertices when $t \rightarrow \infty$ (see e.g. \cite{B09}).
But Theorem~\ref{thm:main} only shows that for each $e\geq 3$, our construction provides families of $t$-e.c. graphs with $O(t^{2e}2^{2et})$ vertices when $t \rightarrow \infty$.
So it also would be interesting to improve Theorem~\ref{thm:main}.
\subsection*{Acknowledgements}
First, we would like to appreciate Masanori Sawa for introducing the concept of $t$-e.c. graphs.
We also would like to give special thanks to Gary Greaves for his comments leading to the remark on expander graphs in Section~\ref{sec:conclusion}.
At last, we would like to appreciate Sanming Zhou for telling us the paper~\cite{LZ15} and careful reading this note.
The author is supported by Grant-in-Aid for JSPS Fellows 18J11282 of the Japan Society for the Promotion of Science.



\begin{thebibliography}{99}
\bibitem{AS16}
N.~Alon, J. H.~Spencer, \newblock
The Probabilistic Method. \newblock
Fourth edition, John Wiley \& Sons, Inc., 2016.

\bibitem{A94}
N.~Alon, Explicit Ramsey graphs and orthonormal labelings,
\emph{Electron. J. Combin.}, {\bf 1}(1994), R12.

\bibitem{AC06} 
W. Ananchuen, L. Caccetta,
Cubic and quadruple Paley graphs with the $n$-e.c. property,
\emph{Discrete Math.}, {\bf 306}(2006), 2954–-2961. 

\bibitem{B10} 
N. de Beaudrap, 
On restricted unitary Cayley graphs and symplectic transformations modulo $n$,
\emph{Electron. J. Combin.}, {\bf 17}(2010), R69. 

\bibitem{BEW98}
B. Berndt, R. J. Evans, K. S. Williams, 
Gauss and Jacobi Sums, 
John Wiley \& Sons, Inc., 1998.

\bibitem{BL06} 
Y. Bilu, L. Linial, Lifts, discrepancy and nearly optimal spectral gap,
\emph{Combinatorica}, {\bf 26}(2006), 495--519. 

\bibitem{BH79}
A. Blass, F. Harary, 
Properties of almost all graphs and complexes, 
\emph{J. Graph Theory}, {\bf 3}(1979), 225--240. 

\bibitem{BEH81}
A.~Blass, G.~Exoo, F.~Harary, 
Paley graphs satisfy all first-order adjacency axioms,
\emph{J. Graph Theory}, {\bf 5}(1981), 435--439. 

\bibitem{BT81}
B. Bollob\'{a}s, A. Thomason, 
Graphs which contain all small graphs, 
\emph{European J. Combin.}, {\bf 2}(1981), 13--15. 

\bibitem{B01}
B. Bollob\'{a}s, 
Random Graphs,
Second edition, Cambridge University Press, 2001.


\bibitem{B09}
A.~Bonato, 
The search for $N$-e.c. graphs, 
\emph{Contrib. Discrete Math.}, {\bf 4}(2009), 40--53. 


\bibitem{CS02}
P.~J.~Cameron, D.~Stark, 
A prolific construction of strongly regular graphs with the $n$-e.c. property,
\emph{Electron. J. Combin.}, {\bf 9}(2002), R31.

\bibitem{CGW89}
F.~R.~K.~Chung, R.~L.~Graham, R.~M.~Wilson, 
Quasi-random graphs,
\emph{Combinatorica}, {\bf 9}(1989), 345--362.



\bibitem{ER63}
P.~Erd\H{o}s, A.~R\'enyi, 
Asymmetric graphs,
\emph{Acta Math. Acad. Sci. Hungar.}, {\bf 14}(1963), 295--315.

\bibitem{ES71} P. Erd\H{o}s, J. Spencer, 
Imbalances in $k$-colorations,
\emph{Networks}, {\bf 1}(1971/72), 379--385.

\bibitem{GS71} 
R. L. Graham, J. H. Spencer, 
A constructive solution to a tournament problem,
\emph{Canad. Math. Bull.}, {\bf 14}(1971), 45--48.

\bibitem{HLW06} 
S. Hoory, N. Linial, A. Wigderson, 
Expander graphs and their applications,
\emph{Bull. Amer. Math. Soc. (N.S.)}, {\bf 43}(2006), 439--561. 

\bibitem{KP04}
A. Kisielewicz, W. Peisert, Pseudo-random properties of self-complementary symmetric graphs,
\emph{ J. Graph Theory}, {\bf 47}(2004), 310--316. 

\bibitem{KRSSS07} 
Y. Kohayakawa, V. R\"{o}dl, M. Schacht, P. Sissokho, J. Skokan, 
Tur\'{a}n's theorem for pseudo-random graphs,
{\it J. Combin. Theory Ser. A }, {\bf 114}(2007), 631--657.

\bibitem{KS06} 
M.~Krivelevich, B.~Sudakov, 
Pseudo-random graphs,
In \emph{More sets, graphs and numbers}, Springer-Verlag, 2006, 199--262. 


\bibitem{LZ15}
X.~Liu, S.~Zhou, 
Quadratic unitary Cayley graphs of finite commutative rings,
\emph{Linear Algebra Appl.}, {\bf 479}(2015), 73--90. 

\bibitem{S19}
S.~Satake, 
On expander Cayley graphs from Galois rings, 
arXiv:1902.03423.


\bibitem{S76}
W.~M.~Schmidt, 
Equations over Finite Fields : An Elementary Approach,
Springer-Verlag, 1976.



\bibitem{T87a}
A.~Thomason, 
Pseudorandom graphs,
\emph{Ann. Discrete Math.}, {\bf 33}(1987), 307--331. 

\bibitem{T87s}
A.~Thomason, \newblock
Random graphs, strongly regular graphs and pseudorandom graphs. \newblock 
In \emph{Surveys in combinatorics 1987}, Volume 123 of 
\emph{London Math. Soc. Lecture Note Ser.}, Cambridge Univ. Press, 1987, 173--195. 

\bibitem{W12}
Z.-X.~Wan, \newblock
Finite Fields and Galois Rings. \newblock
World Scientific Publishing Co. Pte. Ltd., 2012.


\end{thebibliography}
\end{document}